\documentclass[reqno, 12pt]{amsart}
\usepackage[utf8]{inputenc}

\usepackage[margin=1in]{geometry}
\usepackage[shortlabels]{enumitem}
\usepackage{amsmath,amscd,amsthm,amsfonts,amssymb}
\usepackage{eucal}
\usepackage{mathrsfs}
\usepackage{float}
\usepackage{tikz-cd, comment, fancyvrb}
\usepackage{todonotes}
\usepackage{xcolor}

\RequirePackage{color}
\definecolor{myred}{rgb}{0,0,0}
\definecolor{mygreen}{rgb}{0,0.5,0}
\definecolor{myblue}{rgb}{0,0,0.65}

\usepackage[pagebackref]{hyperref}
\hypersetup{citecolor=blue}
\usepackage{tikz}
\usetikzlibrary{matrix,arrows,decorations.pathmorphing}

\theoremstyle{plain}
\newtheorem{theorem}{Theorem}[section]
\newtheorem{proposition}[theorem]{Proposition}

\newtheorem{lemma}[theorem]{Lemma}
\newtheorem*{lemma*}{Lemma}
\newtheorem*{proposition*}{Proposition}

\newtheorem*{truefact*}{Fact}

\theoremstyle{definition}
\newtheorem{definition}[theorem]{Definition}

\theoremstyle{remark}
\newtheorem*{remark}{Remark}

\newcommand{\bb}[1]{\expandafter\newcommand\expandafter{\csname #1\endcsname}{{\mathbb {#1}}}} 

\bb C
\bb R
\bb Q
\bb Z
\bb N

\bb F  
\newcommand{\LL}{\mathcal L}

\renewcommand{\le}{\leqslant}
\renewcommand{\ge}{\geqslant}
\renewcommand{\leq}{\leqslant}
\renewcommand{\geq}{\geqslant}
\renewcommand{\phi}{\varphi}
 
\newcommand{\bfbeta}{\boldsymbol{\beta}}
\newcommand{\bfalpha}{\boldsymbol{\alpha}}

\newcommand{\bfm}{\boldsymbol{m}}
\newcommand{\bfn}{\boldsymbol{n}}
\newcommand{\bfq}{\boldsymbol{q}}
\newcommand{\co}{\textnormal{core}}

\title[ RMF and Making Squares  ]{Random Multiplicative Functions and Making Squares from Polynomial Values}
\author{R\'egis de la Bret\`eche}
\address{ 
Universit\'e   Paris Cité, Sorbonne Universit\'e, CNRS UMR 7586\\ Institut Universitaire de France\\ Institut de Math\'ematiques de Jussieu-Paris Rive Gauche\\
Case Postale 7012\\
F-75251 Paris CEDEX 13\\ France}
\email{regis.delabreteche@imj-prg.fr}

\author{Victor Y. Wang}
\address{Institute of Mathematics, Academia Sinica, Taipei 106319, Taiwan}
\email{vywang@as.edu.tw}

\author{Max Wenqiang Xu}
\address{Yau Mathematical Sciences Center, Tsinghua University, Beijing, China}
\email{maxxu1729@gmail.com}
\thanks{The authors would like to thank
Tim Browning,
C\'edric Pilatte,
and Besfort Shala
for helpful discussions.
We used AI tools to locate technical references,
to catch typos and minor errors,
to draft preliminary versions of Proposition~\ref{large-squares},
and to numerically optimize the exponent of Lemma~\ref{lemma maj CBalphabeta}.
The first author is supported by IUF senior and ANR-FNS Grant ANR-24-CE93-001.
The second author is supported by NSTC grant 114-2115-M-001-010-MY2.
The third author was supported by a Simons Junior Fellowship from the Simons Foundation.}
\dedicatory{ \`A Fouvry pour son  73\`eme anniversaire}
\date{}

\begin{document}

\begin{abstract} 
For a large family of polynomials $P(X)\in \mathbb{Z}[X]$, 
we prove central limit theorems for $\sum_{n\le N} f(P(n))$ for both Rademacher and extended Rademacher multiplicative functions~$f$. To achieve this, we establish a paucity phenomenon in counting solutions to
\[P(n_1)P(n_2)P(n_3)P(n_4) = \square,  \quad 1\le n_1, n_2, n_3, n_4 \le N.\]
Results of Hooley, Evertse--Silverman, and Reuss
play an important role in the proof.
Our estimates are sharpest for $\deg P = 2$,
thanks to the rich theory of Pell--Fermat equations.
\end{abstract}

\maketitle 

\setcounter{tocdepth}{1}
\tableofcontents
\setcounter{tocdepth}{3}

\section{Introduction}
\label{intro}

The study of random multiplicative functions has been an active research area. A central question is to study the limiting distribution for the partial sum 
\[\sum_{n\le N} a(n)f(n)\]
where $a(n)$ are some fixed weights and $f(n)$ is a random multiplicative function (RMF). A RMF is defined by first setting $f(p)$ i.i.d random variables with certain distribution $\mu$ and then define $f(n)$ (completely) multiplicatively.
There are three types of RMFs that are commonly studied: Steinhaus, Rademacher and extended Rademacher with the motivation to model different families of deterministic multiplicative functions. The modern study of the limiting distribution of partial sums of RMFs has surprisingly   many interesting connections to many other branches of mathematics.
In particular, many Diophantine problems have naturally arisen in the process of understanding the moments of the partial sums of RMFs.

In this paper, we study the well-known example that $a(n)$ is the indicator function of the polynomial value $P(n)$ for a given fixed polynomial $P(X)\in \mathbb{Z}[X]$
with positive leading coefficient.
It is widely believed that 
\[\sum_{n\le N } f(P(n)) \]
should behave Gaussian unless $P(X)$ is essentially a degree 1 polynomial. 
The modern tools from Martingale theory can fairly easily reduce the probabilistic number theory question to certain Diophantine questions that we must address in this paper.

The case that $f(n)$ is Steinhaus, i.e.~$f(p)$ is uniformly distributed on the complex unit circle and $f(n)$ is completely multiplicative, is completely solved by Klurman--Shkredov--Xu \cite{KSX23} for any given polynomial $P(X)$. The Diophantine problem studied there is more or less counting solutions to 
\[P(n_1)P(n_2) = P(n_3)P(n_4),  \quad 1\le n_1, n_2, n_3, n_4 \le N. \]
 A paucity phenomenon is required, i.e.~the number of solutions to the equation is dominated by the ``diagonal'' solutions 
 where $\{P(n_1),P(n_2)\} = \{P(n_3),P(n_4)\}$.

One can proceed similarly to attack the case $f(n)$ is Rademacher or extended Rademacher. For the Rademacher case, $f(p)$ is an i.i.d. random variable uniformly distributed on $\{1, -1\}$, and $f(n)$ is multiplicatively defined with $f(n)=0$ if $n$ is not square-free (the historical reason for having such a restriction is to model the M\"{o}bius function).
An extended Rademacher RMF is similarly defined but $f(n)$ is defined completely multiplicatively. The Diophantine equation that we need to deal with in both cases is roughly
\begin{equation}
\label{square-product-4}
P(n_1)P(n_2)P(n_3)P(n_4) = \square,  \quad 1\le n_1, n_2, n_3, n_4 \le N
\end{equation}
with the additional condition $\mu^2(P(n_i)) = 1$ in the Rademacher case.
Perhaps surprisingly, this Diophantine problem is significantly harder than the previous one. 

The main task of the paper is to establish the above fourth moment type estimate, but we would like to emphasize one point that even the variance computation is quite nontrivial in the extended Rademacher case. Notice that the perfect orthogonality in Steinhaus case does not hold for extended Rademacher RMF $f$ as
\[\mathbb{E}[f(n)f(m)] = 1_{nm=\square}.\]
Given this, the variance of the partial sum along polynomial values is roughly
\[\mathbb{E}\Big[\Big(\sum_{n\le N} f(P(n))\Big)^{2}\Big] = \sum_{m,n\le N}1_{P(m)P(n)=\square}.\]
Thus the usual simple variance computation of the partial sums of RMF becomes nontrivial in the extended Rademacher case. It requires us to count solutions to 
\[P(n_1)P(n_2) = \square, \quad 1\le n_1, n_2 \le N.\]
This is straightforward if $P(n_i)$ are square-free (as $P(n_1)=P(n_2)$), which is automatic in the Rademacher case but not in the extended Rademacher case. This feature reflects and causes the difference in studying the two cases. 

In this paper, we manage to solve the Diophantine problems for a large family of $P(X)$ and thus 
establish Theorems~\ref{rademacher-P} and~\ref{extended-rademacher-P} below.
Many interesting Diophantine problems remain,
including the analysis of higher moments in the Rademacher and extended Rademacher cases,
even for $\deg{P}=2$.
In the Steinhaus case, a paucity phenomenon for all even moments was established by Wang--Xu \cite{WX}.
The present work builds on methods from
\cite{WX} and the subsequent work of Chinis--Shala \cite{CS25},
including the use of results of Huxley \cite{Huxley} and Bombieri--Pila \cite{BP},
but features additional complications
from large degrees and large square factors,
which are resolved using tools from Diophantine geometry
going back to Hooley \cite{H84},
Evertse--Silverman \cite{ES86},
and Reuss \cite{Reuss},
as well as new combinatorial decompositions of Diophantine point counts
based on the complete graph $K_4$.

For illustration, assuming $\mu^2(P(n_i))=1$ for $1\le i\le 4$,
we might decompose the solution set to \eqref{square-product-4}
based on the size of $\gcd(P(n_3),P(n_4))$.
Now imagine that we label each edge $ij$ of~$K_4$
with the integer $\gcd(P(n_i),P(n_j))$.
The only edge of $K_4$
neither adjacent to vertex $3$ nor~$4$
is the edge $12$.
Thus if $\gcd(P(n_3),P(n_4))$ is very large, say, then in favorable circumstances
we can use \eqref{square-product-4} to deduce that $\gcd(P(n_1),P(n_2))$ is also very large.
On the other hand, if $\gcd(P(n_3),P(n_4))$ is small then congruential information modulo
$\gcd(P(n_i),P(n_j))$ becomes stronger for edges $ij\ne 34$,
and we want to assemble such information in an optimal manner.
The full combinatorial decompositions we use
are a more complicated variant of this idea, so we defer further details to the proofs.

For $\deg P=2$, we develop a new method using results on the size of solutions of Pell--Fermat equations, allowing us to obtain sharper Diophantine estimates than for $\deg P\ge 3$.
In this case, the multiplicative structure of quadratic norm forms
plays an important role.
It would be interesting to see if our results for $\deg{P}\ge 3$
could similarly be improved using ideas from algebraic number theory or geometry.
We hope that Diophantine equations of the sort studied in the present paper
will inspire future advances in Diophantine analysis,
and in analogous questions in other settings.

\subsection*{Main CLT results}

In the Rademacher case, our main CLT result is the following,
which extends the previous nice work of Chinis--Shala \cite{CS25}, where they proved the case where $\deg{P} = 2$
(or if every irreducible factor of $P$ is linear).\footnote{After writing our paper, we learned that a small modification of \cite{CS25}, specifically in \cite[\S3.2.2]{CS25}, would lead to an alternative proof of Theorem~\ref{rademacher-P}.
What was missing for general $P$
is the elementary Ekedahl-type sieve bound
$\#\{(n_1,n_2)\in [1,N]^2: \gcd(P(n_1),P(n_2))>N\} \ll_{P,\varepsilon} N^{1+\varepsilon}$.
This insight is due to Ihor Pylaiev
and was kindly communicated to us by Besfort Shala.
Nonetheless, our proof of Theorem~\ref{rademacher-P}
is different,
and is easier to generalize to the extended Rademacher case.}
\begin{theorem}
\label{rademacher-P}
Let $P\in \Z[X] $ be separable with $\deg{P}\ge 2$
and with positive leading coefficient.
Assume that
\begin{equation}
\label{count-square-free-values}
\liminf_{N\to \infty} \frac{\#\{1\le n\le N: \mu^2(P(n))=1\}}{N} > 0.
\end{equation}
Let $f$ be a Rademacher RMF.
Then as $N$ tends to $+\infty$,
\[ \frac{1}{\sqrt{\#\{1\le n\le N: \mu^2(P(n))=1\}}} \sum_{1\le n \le N} f(P(n))
\xrightarrow[]{d} \mathcal{N}(0,1) \]
where $\mathcal{N}(0,1)$ stands for a Gaussian distribution with mean~$0$ and variance~$1$. 
\end{theorem}

\begin{proof}
This follows from Chinis--Shala \cite{CS25} if $\deg{P} = 2$
(or if every irreducible factor of $P$ is linear),
and from Proposition~\ref{unconditional-Rademacher-off-diagonal} of the present paper
if $\deg{P} \ge 3$.
\end{proof}

\begin{remark}
Under what conditions does \eqref{count-square-free-values} hold?
First, we must assume that $P$ is \emph{admissible},
meaning that for every prime $p$ there exists $n\in \Z$ such that $p^2\nmid P(n)$.
Now, if every irreducible factor of $P$ has degree $\le 3$,
then \eqref{count-square-free-values} follows from
\cite[Theorem~1.2]{BookerBrowning},
a result building on the work of many authors,
including Hooley \cite{H67} and Reuss \cite{Reuss}.
Under the ABC Conjecture, \eqref{count-square-free-values} holds for all admissible $P$,
thanks to Granville \cite{GranvilleABC}.
Alternatively, it is implicit in \cite[\S4, proof of Theorem~1]{GranvilleABC}
that \eqref{count-square-free-values} would follow from the estimate\footnote{For the reader's convenience, we sketch Granville's reduction of \eqref{count-square-free-values} to \eqref{soft-SFSC}.
First, \cite[\S4, Proposition~1]{GranvilleABC}
gives an unconditional asymptotic formula for
$A(N) := \#\{1\le n\le N: p^2\mid P(n) \Rightarrow p>N\}$,
with main term of order $N$.
Second, $0\le A(N) - \#\{1\le n\le N: \mu^2(P(n))=1\} = o(N)$
by \eqref{soft-SFSC}.}
\begin{equation}
\label{soft-SFSC}
\lim_{N\to \infty} \frac{\#\{1\le n\le N: \exists p>N,\; p^2\mid P(n)\}}{N}
= 0,
\end{equation}
which is equivalent to
what Miller \cite[\S2.2]{Miller} calls the \emph{Square-Free Sieve Conjecture}.\footnote{Miller states the conjecture
in the form $\#\{N< n\le 2N: \exists p>\log{N},\; p^2\mid P(n)\} = o(N)$,
but this is easily shown to be equivalent to
$\#\{N<n\le 2N: \exists p>2N,\; p^2\mid P(n)\} = o(N)$.
Therefore, \eqref{soft-SFSC}
implies Miller's statement.
On the other hand, Miller's statement implies \eqref{soft-SFSC}
by dyadic summation over $n$.}
Finally, we expect that an unconditional Polynomial Rademacher CLT
could now be obtained for all admissible polynomials $P$
over a function field $\mathbb{F}_q(t)$,
using \cite{Ramsay} (see also \cite{Poonen}).
\end{remark}

In the extended Rademacher case, we have the following unconditional result,
which
for a large class of polynomials
confirms the Gaussian prediction of Chinis--Shala from \cite[\S1.4]{CS25}.

\begin{theorem}
\label{extended-rademacher-P}
Let $P\in \Z[X] $ be separable with $\deg{P}\ge 2$
and with positive leading coefficient.
Assume that every irreducible factor of $P$ has degree $\le 3$.
Let $g$ be an extended Rademacher RMF.
Then as $N\to +\infty$,
\[\frac{1}{\sqrt{N}} \sum_{1\le n \le N} g(P(n)) \xrightarrow[]{d} \mathcal{N}(0,1). \]
\end{theorem}

\begin{proof}
If $\deg{P}=2$, this follows from Theorem~\ref{thm: CLT}.
If $\deg{P}\ge 3$, this follows from
Theorem~\ref{THM:conditional-CLT}
and Proposition~\ref{large-squares}.
\end{proof}

\begin{remark}
The statement of Theorem~\ref{extended-rademacher-P}
is still valid if we replace the set of integers $n\in \N$ such that $1\leq n\leq N$ by a subset $S\subset [1,N]$ and $\sqrt{N}$ by $\sqrt{|S|}$,
provided that
there exists a constant $c\in (0,1)$ for which
\begin{equation}
\label{thinner}
|S|\gg N\exp\{ -c\,  \sqrt{(\log N)(\log\log N)}\}.
\end{equation}
\end{remark}

\begin{remark}
There is some hope of extending Theorem~\ref{extended-rademacher-P}
to cases where $P$ has an irreducible factor of degree $\ge 4$.
This will be discussed after Theorem~\ref{THM:conditional-CLT}.
Currently, in the degree $\ge 4$ case, our proofs rely on a
power-saving version, \eqref{SFSC}, of the Square-Free Sieve Conjecture.
In~\S\ref{intro}, we have chosen to focus on unconditional results.
\end{remark}

As previously mentioned, the limiting distribution of $\sum_{1\le n\le N} f(P(n))$ or $\sum_{1\le n\le N} g(P(n))$ is Gaussian if a certain Diophantine counting problem is solved. This deduction step relies on a version of the McLeish Martingale central limit theorem, the details of which are given in~\S\ref{general-clt}.
After \S\ref{general-clt}, we gradually build up to the solution of the Diophantine problem.

As we mentioned earlier, our Diophantine counting results are sharpest for $\deg P = 2$.
The interested reader may consult Propositions~\ref{prop N2}
and~\ref{prop N4} for details on the equations
$P(n_1)P(n_2) = \square$ and $P(n_1)P(n_2)P(n_3)P(n_4) = \square$, respectively.
Weaker Diophantine estimates would suffice
for proving Theorem~\ref{extended-rademacher-P},
but we consider these equations interesting in their own right.
When $\deg P \ge 3$ and every irreducible factor of $P$ has degree $\le 3$,
an analog of Proposition~\ref{prop N2} (with a weaker error term)
may be extracted from the proof of Theorem~\ref{extended-rademacher-P},
but it remains open to prove an analog of Proposition~\ref{prop N4}
without any restriction on square divisors of $P(n_i)$, $1\le i\le 4$.

\begin{remark}
Paucity results for the equations
$P(n_1)P(n_2) = \square$ and $P(n_1)P(n_2)P(n_3)P(n_4) = \square$
do not seem to be easily susceptible to
a direct application of
existing upper-bound sieves.
\end{remark}

\subsection*{Notation}

For any integer $n \ge 1$, we define the \emph{square-free core} of $n$, denoted $\co(n)$, as the unique square-free integer such that $n = \co(n) \cdot k^2$ for some integer $k$.
Equivalently, if $k^2$ is the largest perfect square dividing $n$,
then $\co(n) := n/k^2$.
This is also equivalent to the definition
$\co(n) := \prod_{v_p(n)\textnormal{ odd}} p$
given after \eqref{core-fibers}.
We remark that
$$n/\co(n)=1 \Longleftrightarrow   \mu^2(n)=1.$$
Let $P^+(n)$ be the largest prime factor of $n$, with the convention $P^+(1)=1$.
Let \begin{equation}
    \label{def P*}
P^*(n) := P^+(\textnormal{core}(n)).\end{equation}
Note that $P^*(n) = 1$ if and only if $n = \square$.

Throughout this paper,
if $P(n)\le 0$ then we let $f(P(n))=g(P(n))=\mu^2(P(n))=0$,
for convenience.
That is, we focus mainly on positive values of polynomials $P\in \Z[X]$.

For integers $N\ge 1$, we let $[N] := [1,N] := \{ 1,\ldots, N\}$.

\section{General CLT criteria for Rademacher and Extended Rademacher cases}
\label{general-clt}

We state the central limit theorem for the Rademacher case and the extended Rademacher case.
For a sequence $(a(m))_{m\in \N}$ and $M\ge 1$, we consider
 the $L^2$ norms
defined by 
$$|| (a_m)_{m\le M} ||_2^2:=\sum_{1\le m\le M} a(m)^{2} $$
and
$$|| a||_2^2:=\sum_{  m\ge 1} a(m)^{2}.$$

To prove a central limit theorem for partial sums of random multiplicative functions, one standard way is to verify the following criteria, which is essentially an application of McLeish central limit theorem, pioneered by Harper in \cite{Harper13}
and generalized by Soundararajan--Xu \cite[Theorem~3.1]{SX23}.
For a different application, the property of martingale difference was first observed by Basquin \cite{Basquin} and used in Lau--Tenenbaum--Wu \cite{LTW}. 
The following is the Rademacher version of \cite[Theorem~3.1]{SX23}.
\begin{theorem}[Rademacher]\label{thm: rad}
Let $a(m)$ be real and $f(m)$ be a Rademacher random multiplicative function.
Assume that $a(m)$ is supported on square-free integers $m$;
that is, $a(m)\ne 0\Rightarrow \mu^2(m)=1$.
If
there exists a subset $S=S_M\subseteq [2, M]$ such that 
    \begin{enumerate}
        \item \[\sum_{m\in [1, M] \smallsetminus S} a(m)^{2} = o\big(|| (a_m)_{m\le M} ||_2^2\big) \]
        \item \[\sum_{\substack{m_1, m_2, m_3, m_4 \in S \\ m_1m_2m_3m_4 =\square\\ P^{+}(m_1)=P^{+}(m_2) = P^{+}(m_3)=P^{+}(m_4)}}  a(m_1)a(m_2)a(m_3)a(m_4)  =  o \big( || (a_m)_{m\le M} ||_2^4 \big) \]
        \item \[\sum_{\substack{m_1, m_2, m_3, m_4 \in S \\ m_1m_2m_3m_4 =\square}}  a(m_1)a(m_2)a(m_3)a(m_4)  =  (3+o(1)) \big( || (a_m)_{m\le M} ||_2^4 \big) . \]
    \end{enumerate} 
    Then as $M\to +\infty$,
    \[ \frac{1}{   || (a_m)_{m\le M} ||_2   } \sum_{1\le m \le M} a(m) f(m) \xrightarrow[]{d} \mathcal{N}(0,1). \]
\end{theorem}

\begin{proof}
This is implicit in \cite[\S9]{SX23}.
It is a routine modification of \cite[Theorem~3.1]{SX23}.
\end{proof}

We next state the version for extended Rademacher RMF for our particular case. 
Fix an arbitrary polynomial $P\in \Z[X]$ with positive leading coefficient.
Given $N$, we write 
\begin{equation}
\label{core-fibers}
a(m) = a_N(m) : = \big|\big\{1\le n\le N: P(n)>0 \textnormal{ and } \co(P(n)) = m\big\}\big|. 
\end{equation}
Here $\co(\ell) := \prod_{v_p(\ell)\textnormal{ odd}} p$
is the ``reduction of the positive integer $\ell$ mod squares'',
which is sometimes (ambiguously) called the square-free part of $\ell$.
Note that if $m$ is not the core for any $P(n)$ with $1\le n \le N$, then $a(m)=0$. 

Then for an extended Rademacher $g$ and $f= \mu^2g$, we have 
\[\sum_{1\le n \le N} g(P(n)) = \sum_{m\ge 1} a(m) f(m). \]
It has mean value $a(1)$, which will be negligible in practice.
The variance is 
\[\mathbb{V}\Big[\sum_{1\le n \le N} g(P(n))\Big] =  || a ||_2^2 - a(1)^2.\]
Now we can just plug $a(m)$ into Theorem~\ref{thm: rad} to get the following statement.

\begin{theorem}[Extended Rademacher for polynomial values]\label{theorem Extended Rademacher}
Fix $P\in \Z[X]$ as above.
    Let $a(m) = a_N(m)$ be defined as above, and $g$ be an extended Rademacher random multiplicative function. Suppose there is a choice of a set $S=S_N \subseteq [2,\infty)$, such that
    \begin{enumerate}
        \item 
        \[\sum_{m\in \mathbb{Z} \smallsetminus S} a(m)^{2} = o\big( || a ||_2^2\big) \]
        \item \[\sum_{\substack{m_1, m_2, m_3, m_4 \in S \\ m_1m_2m_3m_4 =\square\\ P^{+}(m_1)=P^{+}(m_2) = P^{+}(m_3)=P^{+}(m_4)}}  a(m_1)a(m_2)a(m_3)a(m_4)  =  o \big(  || a ||_2^4 \big) \]
        \item \[\sum_{\substack{m_1, m_2, m_3, m_4 \in S \\ m_1m_2m_3m_4 =\square}}  a(m_1)a(m_2)a(m_3)a(m_4)  =  (3+o(1))|| a ||_2^4. \]
    \end{enumerate}
        Then as $N\to +\infty$, 
    \[\frac{1}{ || a ||_2} \sum_{1\le n \le N} g(P(n)) \xrightarrow[]{d} \mathcal{N}(0,1). \]
\end{theorem}

This follows directly from Theorem~\ref{thm: rad}. 
    We note that our statement is for a polynomial-valued set $\{P(n)\}_{1\le n \le N}$, but it holds for an arbitrary set $\{\alpha(n)\}_{1\le n \le N}$ with suitable and straightforward modifications.

\section{Background on Diophantine geometry and sieves}

\begin{lemma}
\label{CS2.5}
Let $N\ge 1$ and $\varepsilon > 0$.
Let $P\in \mathbb{Z}[X]$ be separable with $\deg{P} \ge 2$,
where every coefficient of $P$ has absolute value
$\le \varepsilon^{-1} N^{\varepsilon^{-1}}$.
Let $1\le a,b \le \varepsilon^{-1} N^{\varepsilon^{-1}}$ with $a\ne b$.
Then
\[
\#\{(n_1,n_2)\in [N]^2: aP(n_1) - bP(n_2) = 0\}
  \ll_{\varepsilon,\deg{P}} N^{\psi(P)+\varepsilon},
\]
where $\psi(P) := 0$ if $\deg{P}=2$ and $\psi(P) := 1/3$ if $\deg{P}\ge 3$.
\end{lemma}

\begin{proof}
For $\deg{P}\ge 3$, this follows from Chinis--Shala \cite[Lemma~2.5]{CS25},
which builds on Bombieri--Pila \cite{BP}.
Their proof also gives the stronger $N^\varepsilon$ bound
when $\deg{P} = 2$.
\end{proof}

\begin{remark}
Lemma~\ref{CS2.5} can sometimes be improved.
Suppose $P\in \mathbb{Z}[X]$ is separable with $\deg{P} = p \ge 5$ prime.
For all $a,b\ge 1$ with $a\ne b$,
the polynomial $F(X,Y) := aP(X) - bP(Y)$
is irreducible over $\Q$,
since its leading homogeneous part $ac_PX^p - bc_PY^p$
has factorization type $(p)$ or $(p-1,1)$ over $\Q$,
and $F$ has no linear factor over $\C$ by \cite[Lemma~2.5]{WX}.
By Galois theory, $F$ is absolutely irreducible,
whence the bound $N^{1/3+\varepsilon}$ can be replaced by $N^{1/p+\varepsilon}$.
\end{remark}

\begin{lemma}
\label{weak-quadratic}
Let $P\in \mathbb{Z}[X]$ be separable with $\deg{P} = 2$.
For all $N,d\ge 1$ and $\varepsilon>0$,
we have $\#\{n\in [N]: \exists~ m\in \Z,\; P(n) = dm^2\} \ll_{P,\varepsilon} N^\varepsilon$.
\end{lemma}

\begin{proof}
After completing the square in $n$,
this follows from Vaughan--Wooley \cite[Lemma~3.5]{VaughanWooleyActaMath1995},
a result building on Hua \cite[Chapter~11]{Hua}.
\end{proof}

\begin{remark}
The bound $O_{P,\varepsilon}(N^\varepsilon)$
can be improved.
In \eqref{bound car xy}, we will prove
a bound of the shape $O_P(\frac{\log N}{1+\log d} + 1)$ for square-free $d$.
However, we prefer to defer this to when we really need it.
\end{remark}

Lemma~\ref{siegel} is a uniform version of Siegel's theorem on integral points.
It has been implicitly used before by Granville \cite[proof of Corollary~2]{GranvilleABC}.
Its proof relies on Evertse--Silverman \cite[Theorem~1(b)]{ES86},
which Gemini 3.1 Pro suggested to us in an incorrectly stated form.
After consulting related work \cite{Alpoge,BombieriGubler},
we wrote a proof using
Bombieri--Gubler \cite[Theorem~5.3.5]{BombieriGubler},
which required
genus theory such as \cite[Theorem~1]{Lemmer} or \cite[\S4]{BhargavaPlus}.
Finally, GPT-5.5 Pro simplified the proof,
using a correctly stated form of \cite[Theorem~1(b)]{ES86}.

\begin{lemma}
\label{siegel}
Let $P\in \mathbb{Z}[X]$ be separable with $\deg{P} \ge 3$.
For every $d\ge 1$,
the curve $P(x) = dy^2$ has at most $2^{O_P(1+\omega(d))}$
points $(x,y)\in \mathbb{Z}^2$.
\end{lemma}

\begin{proof}
First observe that
\begin{equation*}
\#\{(x,y)\in \mathbb{Z}^2:
P(x) = dy^2\}
\le \#\{(x,z)\in \mathbb{Z}^2:
z^2 = dP(x)\},
\end{equation*}
since each pair $(x,y)$ on the left
gives rise to a pair $(x,z) = (x,dy)$ on the right.
But Evertse--Silverman \cite[Theorem~1(b)]{ES86},
with $K=\Q$,
$S = \{p\mid \textnormal{disc}(dP)\} \cup \{\infty\}$,
$f(X) = dP(X)$,
and $L\subset \C$ a number field containing three roots of $P(X)$,
shows that
\begin{equation*}
\#\{(x,z)\in \mathbb{Z}^2:
z^2 = dP(x),
\; z\ne 0\}
\le 2\cdot 7^{[L:\Q](4+9|S|)} \cdot 4^{\textnormal{rank}(\textnormal{Cl}_L[2])},
\end{equation*}
where $\textnormal{Cl}_L$ is the class group of $L$.
The field $L$ depends only on $P$, not on $d$.
Since
\begin{equation*}
|S| = 1 + \omega(\textnormal{disc}(dP))
\le 1 + \omega(d) + \omega(\textnormal{disc}(P))
\ll_P 1 + \omega(d),
\end{equation*}
the lemma immediately follows.
\end{proof}

\begin{remark}
A conjecture of Granville \cite[Conjecture~1.5(i)]{GranvilleTwists}
would imply that if $\deg{P}$ exceeds some absolute constant,
then the bound $2^{O_P(1+\omega(d))}$
in Lemma~\ref{siegel}
can be improved to $O_P(1)$, uniformly over $d\ge 1$.
GPT-5.5 Pro helped us locate this reference.
\end{remark}

We will also need results on the paucity of polynomial values with large square divisors.
We begin by recalling the following easy result.

\begin{lemma}
\label{easy-sieve}
Let \(P\in \mathbb Z[X]\) be separable with $\deg{P}\ge 2$.
If $N,M_1,M_2\ge 1$ and $\varepsilon>0$ then
\begin{equation*}
\mathcal{B}(N,M_1,M_2):= \#\{n\in [N]: \exists m\in (M_1,M_2],\; m^2\mid P(n)\}
\ll_{P,\varepsilon} \frac{N}{M_1^{1-\varepsilon}}
+ M_2^{1+\varepsilon}.
\end{equation*}
\end{lemma}

\begin{proof}
Let
\begin{equation}
    \label{def rhoP}
 \rho_P(q):=\#\{\nu\bmod q:P(\nu)\equiv0\bmod q\}.
\end{equation}
Since \(P\) is separable with $\deg{P}\ge 2$,
it follows from
Huxley \cite{Huxley}
that
\begin{equation}
\label{rho-bound}
 \rho_P(q)\le |\textnormal{disc}(P)|^{1/2} (\deg{P})^{\omega(q)}.
\end{equation}
(Weaker versions were known before Huxley;
see Nagell \cite[pp.~87--90]{NagellBook}.)
Therefore,
\begin{equation*}
\mathcal{B}(N,M_1,M_2) \le \sum_{M_1<m\leq M_2} \rho_P(m^2) \left(\frac{N}{m^2} + 1\right)
\ll_{P,\varepsilon} \frac{N}{M_1^{1-\varepsilon}}
+ M_2^{1+\varepsilon}.
\qedhere
\end{equation*}
\end{proof}

\begin{definition}
Let $P\in \Z[X]$.
Let $$\mathcal{B}(N,M) := \#\{n\in [N]: \exists m>M,\; m^2\mid P(n)\}.$$
Fix $\delta>0$.
If
\begin{equation}
\label{SFSC}
\mathcal{B}(N,N^{1/2}) \ll_{P,\delta} N^{1-\delta}
\end{equation}
for all $N\ge 1$,
then we say that \emph{H$\delta$ holds}.
\end{definition}

The following result builds on
work of Booker--Browning \cite[\S4]{BookerBrowning}.
We identified the proof strategy, and
GPT-5.5 Pro assisted in drafting an initial version of the proof,
with a worse exponent $\delta_0$.
We then checked and substantially revised the argument,
especially in the following aspects:
improving the treatment of the ranges $I_2$ and $I_3$ in \cite[\S4]{BookerBrowning},
and simplifying the subsequent analysis of composite square divisors.

\begin{proposition}
\label{large-squares}
Let \(P\in \mathbb Z[X]\) be separable with $\deg{P}\ge 2$.
Assume that every irreducible factor of \(P\) has degree at most \(3\).
If \(\delta_0<1/13\),
then H$\delta_0$ holds.
\end{proposition}

\begin{proof}
Fix \(0<\delta_0<1/13\). Choose \(\vartheta\) such that
\[
  \delta_0<\vartheta<\min\left\{\frac1{11},\,\frac{1-11\delta_0}{2}\right\}.
\]
Write
$P=c\prod_{j=1}^J P_j$
with \(c\in\mathbb Z\setminus\{0\}\) and with the \(P_j\in\mathbb Z[X]\)
primitive, pairwise non-proportional, irreducible polynomials. Let
\[
 \Sigma_2(N;\vartheta)
 :=
 \#\{N/2<n\leq N:\exists p>N^\vartheta\ {\rm prime},\ p^2\mid P(n)\}.
\]
As in \cite[\S4]{BookerBrowning}, split $(N^\vartheta,\infty)$ into the intervals
\[
I_1=(N^\vartheta,N^{1-\vartheta}],\qquad
I_2=(N^{1-\vartheta},N^{1+\vartheta}],\qquad
I_3=(N^{1+\vartheta},\infty).
\]
We will first show that
\begin{equation}
\label{large-prime-square-tail}
 \Sigma_2(N;\vartheta)\ll_{P,\delta_0}N^{1-\delta_0}.
\end{equation}
The contribution from $P(n)=0$ is trivially $O_P(1)$,
so in the proof of \eqref{large-prime-square-tail} we assume $P(n)\ne 0$ whenever convenient.
Similarly, we assume $N$ to be large whenever convenient.

The contribution from \(I_1\) is exactly the argument of
\cite[Lemma~4.1]{BookerBrowning}:
\[
 \Sigma_{2,1}\ll_P
 \sum_{N^\vartheta<p\leq N^{1-\vartheta}}
 \left(\frac{N}{p^2}+1\right)
 \ll_P N^{1-\vartheta}.
\]
Since \(\vartheta>\delta_0\), this is \(O(N^{1-\delta_0})\).
Alternatively, this would follow from our Lemma~\ref{easy-sieve}.

We next treat \(I_2\). For all sufficiently large \(N\), every
\(p\in I_2\) avoids the finitely many primes dividing \(c\), the discriminants
of the \(P_j\), and the pairwise resultants of the \(P_j\).
Hence \(p^2\mid P(n)\)
implies \(p^2\mid P_j(n)\) for some \(j\).
Since $P_j(n)\ne 0$, we have $|P_j(n)| \ge p^2$.
Thus $\deg{P_j}\ge 2$, since $N$ is large.
The contribution from a quadratic factor $P_j$ is
\begin{equation*}
\ll_{P,\varepsilon} \sum_{0<|d|\ll_P N^{2\vartheta}} N^\varepsilon
\ll_P N^{2\vartheta+\varepsilon},
\end{equation*}
where we have written $P_j(n) = p^2 d$ with $n\le N$ and $p\in I_2$,
and bounded the number of possible pairs $(n,p)$ using Lemma~\ref{weak-quadratic}.
This is \(O(N^{1-\delta_0})\), since \(\vartheta<1/11\).

It remains in the range $I_2$ to consider cubic factors $P_j$.
Dyadic decomposition,
similar to that in \cite[proof of Lemma~4.3]{BookerBrowning},
reduces the count for a cubic \(P_j\) to
\[
 \ll_{P,\varepsilon}
 N^\varepsilon
 \max_{A,B}
 \mathcal N(N;A,B),
\]
where in the notation of Reuss
\[
 \mathcal N(N;A,B)
 :=
 \#\{(n,a,b): n\asymp N,\ a\asymp A,\ b\asymp B,\ \mu^2(a)=1,\ P_j(n)=a^2b\}.
\]
Here \(a\) corresponds to \(p\), and $P_j$ is cubic, so we may assume that
\[
N^{1-\vartheta} \ll A \ll N^{1+\vartheta},
\qquad
 A^2B\asymp_P N^3.
\]
By Reuss's determinant-method estimate \cite{Reuss}, as quoted in
\cite[proof of Lemma~4.3]{BookerBrowning}, if
\[
 (AB)^{3/11}N^{-2/11}\ll M\ll \min\{N^{1/2},A^{1/2}\}
\]
then
\[
 \mathcal N(N;A,B)\ll_{P,\varepsilon} M^{2/3}N^{2/3+\varepsilon}.
\]
Using \(A^2B\asymp N^3\), we have
\[
 (AB)^{3/11}N^{-2/11}
 \asymp N^{7/11}A^{-3/11}
 \ll N^{4/11+3\vartheta/11},
\]
since \(A\gg N^{1-\vartheta}\).
Thus we may take
\[
 M=N^{4/11+3\vartheta/11};
\]
this is admissible because \(\vartheta<1/11\). Consequently
\[
 \mathcal N(N;A,B)
 \ll_{P,\varepsilon}
 N^{10/11+2\vartheta/11+\varepsilon}.
\]
Therefore
\[
 \Sigma_{2,2}
 \ll_{P,\varepsilon}
 N^{10/11+2\vartheta/11+\varepsilon}.
\]
By the choice of \(\vartheta\) and then of \(\varepsilon\), this is
\(O(N^{1-\delta_0})\).

Finally consider \(I_3\). Again, for all sufficiently large \(N\), if
\(p\in I_3\) and \(p^2\mid P(n)\), then \(p^2\mid P_j(n)\) for some \(j\).
Then \(\deg P_j=3\), since $P_j(n)\ne 0$.
Write
$P_j(n)=p^2d$.
Then $|d|\ll_P N^{1-2\vartheta}$.
Bounding the number of possible pairs \((n,p)\) using Lemma~\ref{siegel}, we get
\[
 \Sigma_{2,3}
 \ll_P
 \sum_j\sum_{0<|d|\ll_P N^{1-2\vartheta}}2^{C_P\omega(d)}
 \ll_{P,\varepsilon} N^{1-2\vartheta+\varepsilon}.
\]
Since \(2\vartheta>\delta_0\), this is also \(O(N^{1-\delta_0})\). This proves
\eqref{large-prime-square-tail}.

We are now ready to prove Proposition~\ref{large-squares}.
Let $m_n^2$ be the largest square divisor of $P(n)$,
where $m_n=\infty$ if $P(n)=0$.
Let $P^+(\infty)=\infty$.
Dyadic summation of \eqref{large-prime-square-tail} over $N$ gives
$$\#\{n\in[N]: P^+(m_n)>N^\vartheta\}\ll_{P,\delta_0} N^{1-\delta_0}.$$
It remains to count \(n\) for which
$m_n>N^{1/2}$ and $P^+(m_n)\leq N^\vartheta$.
Considering partial products of
prime factors of \(m_n\), with multiplicity, we obtain a divisor \(m'\mid m_n\) such that
\[
 N^{1/2-\vartheta}<m'\leq N^{1/2}.
\]
Since \((m')^2\mid P(n)\), it follows from Lemma~\ref{easy-sieve} that
\[
\#\{n\in[N]: m_n>N^{1/2},\ P^+(m_n)\leq N^\vartheta\}
\ll_{P,\varepsilon}
 N^{1/2+\vartheta+\varepsilon}
 \ll N^{1-\delta_0},
\]
since \(\vartheta<1/4\) and \(\vartheta>\delta_0\).
Combining the two contributions gives the proposition.
\end{proof}

\section{General correlation bounds}

We prove two lemmas that will let us
discard small sets in moment calculations.

\begin{lemma}
\label{mixed-2-correlation}
Let \(P\in \mathbb Z[X]\) be separable with $\deg{P}\ge 2$.
Let $S \subseteq [N]$.
Let $F(S)$ be the number of pairs $(n_1,n_2)\in S \times [N]$
for which $P(n_1)P(n_2)$ is a nonzero square.
Then $F(S) \ll_{P,\varepsilon} N^\varepsilon |S|$.
\end{lemma}

\begin{proof}
Let $\sigma\in \{\pm 1\}$ be the sign of $P(n_1)$.
Let $d = \co(\sigma P(n_1))$.
Write $\sigma P(n_j) = dm_j^2$ where $m_j\ge 1$.
By Lemma~\ref{weak-quadratic} if $\deg{P}=2$,
or by Lemma~\ref{siegel} if $\deg{P}\ge 3$, we get
\begin{equation*}
F(S)
\ll \sum_{n_1\in S} N^\varepsilon
= N^\varepsilon |S|. \qedhere
\end{equation*}
\end{proof}

\begin{remark}
For some polynomials $P(X)$,
such as $P(X) = a Q(X)^4 \pm 1$ where $a\ge 1$ is an integer
and $Q\in \mathbb{Z}[X]$ with $\deg{Q}\ge 1$,
it is possible to improve the estimate in Lemma~\ref{mixed-2-correlation} to
$F(S) \ll_P |S|$.
The reason is that for such polynomials we may improve the
Diophantine bound in Lemma~\ref{siegel}
to $O_P(1)$, uniformly over $d\ge 1$.
See a result by Akhtari \cite{Akhtari}.
\end{remark}

\begin{lemma}
\label{mixed-4-correlation}
Let $S_i \subseteq [N]$ for $1\le i\le 4$.
Let $E(S_1,S_2,S_3,S_4)$ be the number of tuples
$(n_1,n_2,n_3,n_4)\in \prod_{1\le i\le 4} S_i$
for which $P(n_1)P(n_2)P(n_3)P(n_4)$ is a nonzero square.
Then
$$E(S_1,S_2,S_3,S_4) \le \prod_{1\le i\le 4} E(S_i,S_i,S_i,S_i)^{1/4}.$$
\end{lemma}

\begin{proof}
Let $g(m)$ be an extended Rademacher RMF, supported on $m\ge 1$ as usual.
We extend this to $\Z$ as follows.
Let $h(-1)\in \{\pm 1\}$ uniformly at random.
For all $m\in \Z$, let $$h(m) := g(m) 1_{m>0} + g(-m) h(-1) 1_{m<0}.$$
Then $E(S_1,S_2,S_3,S_4) = \mathbb{E}[\prod_{1\le i\le 4} \sum_{n_i\in S_i} h(P(n_i))]$.
Now use H\"{o}lder's inequality.
\end{proof}

\section{The variance (quadratic case)}

When $P$ is a polynomial of degree $2$ in $\Z[X]$ we study the distribution of points $(m,n_1,n_2)\in \N^3$ that satisfy 
\begin{equation}
    \label{be a square}
P(n_1)P(n_2)=m^2.\end{equation}
We call a solution  {\it diagonal} if it satisfies $P(n_1)=P(n_2)$ or if $P(n_1)P(n_2)=0$.

We introduce the cardinality of off-diagonal solutions
$$M_2(N;P):=\#\left\{ (m,n_1,n_2)\in \N^3\,: \begin{array}{l} 1\leq n_1,n_2\leq N,\\ P(n_1)P(n_2)=m^2\neq 0,\\  n_1 \neq  n_2 \end{array}\right\}.$$

 \begin{proposition}\label{prop N2}
Let $P$ be a separable polynomial of degree $2$ in $\Z[X]$. Then for $N\geq 2$
we have
$$M_2(N;P)\ll_P \sqrt{N}
.$$
     \end{proposition} 

\begin{remark}
  The bound is   optimal up to constant factors,
  since we will prove\begin{equation}
        \label{eq optimal}
        M_2(N;X^2-1)\gg \sqrt{N}.
    \end{equation}
    After writing our paper,
    we learned that partial results toward
    Proposition~\ref{prop N2} have been obtained before.
    This is hinted at in \cite[\S1.4]{CS25}.
    A result for the case $P(X) = X(X+1)$ explicitly appears in
    Problem~10 of the IMC competition 2025,
    due to Besfort Shala.
\end{remark}

\begin{proof}[Proof of Proposition~\ref{prop N2}]
Before starting the proof of Proposition~\ref{prop N2}, we recall some facts about Pell--Fermat equations.
When $d\geq 2$  is square-free, we write $n^2-dm^2=\Delta$ the Pell--Fermat equation and denote by $\varepsilon_{d,\Delta}=n+m\sqrt{d}$ the smallest solution such that $n,m\ge 1$ and $\varepsilon_{d,\Delta}>\sqrt{|\Delta|}$, if it exists,
and $\eta_d$ the fundamental unit of $\Q[\sqrt{d}].$ We have $\eta_d\in {\mathcal O}_d$ the ring of integers with ${\mathcal O}_d\in \big\{\Z[\sqrt{d}], \Z\big[\frac{1+\sqrt{d}}{2}\big]\big\}$.
Following some arithmetical properties of $d$, we have $\eta_d\in \{\varepsilon_{d,1},\varepsilon_{d,-1},\tfrac12\varepsilon_{d,4},\tfrac12\varepsilon_{d,-4} \}.$
We note that $\varepsilon_{d,\Delta}\geq \sqrt{d}$ for any $\Delta$ and $\eta_d\geq\tfrac12 \sqrt{d}$.

Let  ${\mathcal I}_{\Delta,d}$ be the set of fractional principal ideals $I$ of ${\mathcal O}_d$ such that $N(I)=|\Delta|$. We have $|{\mathcal I}_{\Delta,d}|\leq \tau(\Delta)^2\ll_\Delta 1$ where $\tau(\Delta)$ is the number of divisors of $\Delta$.   For any $I\in {\mathcal I}_{\Delta,d}$, there exists $\gamma_I\in {\mathcal O}_d$ such that $\{\gamma\in I\,:\, \langle \gamma\rangle =I\}=\{ \pm \eta_d^{k_I}\gamma_I\,:\quad k_I\in \mathbb Z
\}$.
 If $n^2-dm^2=\Delta$, then
 $$n+m\sqrt{d} \in \bigcup_{I\in \mathcal I_{\Delta,d}}\{ \langle n+m\sqrt{d}\rangle =I\}.$$
 We deduce that for any $N\geq 2$ and   
 square-free $d\geq 2$
\begin{equation}
    \label{bound car xy}
\big|\big\{(m,n)\in \N^2\,:\qquad  n\leq N,\quad  n^2-dm^2=\Delta\big\}\big|\ll_\Delta \frac{\log N}{\log d}+1    .\end{equation}
In fact, we write $n+m\sqrt{d}=\gamma_I\eta_d^k$ and we impose that $\gamma_I$ is minimal with $\gamma_I\geq \sqrt{\Delta}.$
Then to prove \eqref{bound car xy}, it remains to count $k$, which is done using the lower bound $\log\eta_d\gg \log d$.

 In \cite[Theorem 1]{H84}, it is shown that when $D\geq 2$ we have 
  $$S(D):=\#\{ d\leq D\,:\quad \exists (m,n)\in \N^2,\quad n+m\sqrt{d}\leq d,
  \quad n^2-dm^2=1\}\ll \sqrt{D}(\log D)^2.$$
In $S(D)$ the contribution of perfect squares $d$ is clearly $\ll \sqrt{D}. $

 Please, see \cite{F16}, \cite{FJ12} and \cite{FJ13} for further references.
 In the same fashion,
 it is easy to show that for any fixed $\Delta\neq 0$  when $D\geq 2$ we have 
  $$S(D;\Delta):=\#\{ (m,n,d)\in \N^3\,:\quad d\leq D,\quad n+m\sqrt{d}\leq d,\quad n^2-dm^2=\Delta\}\ll_\Delta \sqrt{D}(\log D)^2,$$
  where the basic idea of the proof is to interpret the equality $n^2-dm^2=\Delta$
 as a congruence condition $n^2-\Delta \equiv 0\bmod{m^2}$.
 Although we will not directly use this result below, we will use some ideas from its proof. 
   
\medskip Now we are ready to prove Proposition~\ref{prop N2}.

Let $P(X):=aX^2+bX+c$ with $a\neq 0$. If $a<0$ we replace $P$ with $-P$,
so that we can assume $a\geq 1$. We write $P_1(X):=(2aX+b)^2-(b^2-4ac)$ and $P_2(X):=X^2-\Delta$ with $\Delta:=b^2-4ac$. Since $P_1=4aP=P_2(2aX+b)$ and $P_2(-n) = P_2(n)$, we have
$$M_2(N;P)=M_2(N;P_1)\leq 4M_2(2aN+|b|; P_2) + O_{P,\varepsilon}(N^\varepsilon),$$
by Lemma~\ref{mixed-2-correlation} applied to the set $S = \{-\frac{b}{2a}\}$ if $\frac{-b}{2a}\in \N$.
So we can restrict ourselves to $P(X)=X^2-\Delta$,
     with   $\Delta\neq 0$ since $P$ is separable.
Discarding a bounded finite number of values of $n_1$ or $n_2$ is harmless by Lemma~\ref{mixed-2-correlation}, so we assume $n_j>\sqrt{|\Delta|} $ so that $P(n_j)>0$.
     
For any solution of \eqref{be a square}, there exists a square-free integer $d\in \N$ and two non-negative integers $m_1,m_2\in \N$ such that 
$P(n_1)=dm_1^2$ and $P(n_2)=dm_2^2$; that is,
$$n_1^2-dm_1^2=\Delta,\quad n_2^2-dm_2^2=\Delta.$$ 
If $d=1$, the number of associated solutions is $\ll \tau(\Delta)^2$. So
we can then assume $d\geq 2$ square-free. 
We have
\begin{equation}    
    \label{rapport} n_2+m_2\sqrt d=\frac{ (n_3+m_3\sqrt d)(n_1+m_1\sqrt d)}{\Delta}=w_d(n_1+m_1\sqrt d)
,\end{equation}
with
$$ n_3:=n_1n_2-dm_1m_2, \quad m_3=n_1m_2-m_1n_2,\qquad w_d:=\frac{  n_3+m_3\sqrt d}{\Delta}.$$ 
We have $m_3\neq 0$ since otherwise $n_1 + m_1\sqrt{d}=n_2 +m_2\sqrt{d}.$
As $w_d>0$, by symmetry, we can also assume that $n_3m_3\geq 0.$ This implies that 
$w_d\geq \sqrt{d}/|\Delta|.$
Note that if $\Delta$ is a square, then $(n_1,m_1)=(\sqrt{\Delta},0)$ corresponds to a diagonal solution.
Since \begin{equation}
    \label{ineg}N\gg n_2+m_2\sqrt{d}= w_d(n_1+m_1\sqrt{d})\geq \frac{\sqrt{d}}{|\Delta|}(n_1+m_1\sqrt{d})\geq \frac{ d}{|\Delta|},\end{equation} we deduce $d\ll |\Delta| N.$

We count $(m_1,m_2,n_1,n_2,d)$ and we consider $ n_1+m_1\sqrt{d}$   such that 
$$n_1^2-\Delta=dm_1^2,\qquad 2^k\sqrt{d}<n_1 +m_1 \sqrt{d}\leq 2^{k+1}\sqrt{d}\leq d,\quad D:=2^j<d\leq 2^{j+1}.$$
By \eqref{ineg}, we have $D\leq N|\Delta|/2^k.$
Since $n_1-m_1\sqrt{d} = \Delta/(n_1+m_1\sqrt{d})$,
we have $m_1\asymp 2^k$, $2^k\ll D^{1/2}$ and $n_1\leq 2^{k+3/2}D^{1/2}$.

Using \eqref{bound car xy},
the contribution of $d\leq N^{1/2}$ is easily bounded by
$$\ll \sum_{d\leq N^{1/2}}\mu^2(d)\frac{(\log N)^2}{(\log d)^2+1}\ll \sqrt{N}.$$
So we can assume $  N^{1/2}<d\leq N|\Delta| $ so for any fixed square-free $d$, the number of $(n_2,m_2)$ such that $n_2^2-dm_2^2=\Delta$, $n_2 +m_2 \sqrt{d}\ll N$ is bounded.

For $k$ and $D\gg N^{1/2}$ such that $2^k\ll D^{1/2}$, the number of $(m_1,m_2,n_1,n_2,d)$ is 
$$\ll \sum_{m_1\asymp 2^k}\sum_{\substack{n_1\leq 2^{k+3/2}D^{1/2}\\ n_1^2\equiv \Delta\bmod m_1^2}}1
\ll \sum_{m_1\asymp 2^k} 2^{k+3/2}D^{1/2}\frac{\rho_P(m_1^2)}{m_1^2}\ll D^{1/2}(k+1), $$
where we used \eqref{rho-bound}. So summing $D=2^j$ such that $2^j\leq N|\Delta|/2^k $ and $k\geq 0$, we get a total contribution 
$$\ll N^{1/2}\sum_{k\ge 0} \frac{k+1}{2^{k/2}} \ll  N^{1/2}.$$ 

By the same way, we bound the number of $(m_1,m_3,n_1,n_3,d)$ such that 
$$n_3^2-\Delta^2=dm_3^2,\qquad 2^k\sqrt{d}<|n_3+m_3\sqrt{d}|\leq 2^{k+1}\sqrt{d}\leq d,\quad D:=2^j<d\leq 2^{j+1}.$$
The total contribution is again
$ \ll    N^{1/2}.$

Since we have ruled out cases $n_1+m_1\sqrt{d}\leq d$ and $|n_3+m_3\sqrt{d}|\leq d$, we obtain from \eqref{ineg} and $n_1+m_1\sqrt{d}>d$, $|n_3+m_3\sqrt{d}|> d $ the inequality $d^2\ll N|\Delta|.$
The contribution of this last case is 
$$\ll \sum_{d^2\ll N|\Delta|}\frac{(\log N)^2}{(\log d)^2+1}\ll \sqrt{N}.$$

This completes the proof of Proposition~\ref{prop N2}.
 \end{proof}

\begin{proof}[Proof of lower bound \eqref{eq optimal}]
We now prove \eqref{eq optimal}. Let $P(X):=X^2-1$. First we observe that, for any $\alpha>0,$  we have
$$M_2(N;P)\geq \#\{ d\,:\, \varepsilon_{d,1}\leq \sqrt{N}\}
\geq \#\{ d\leq N^{1/(1+\alpha)}\,:\, \varepsilon_{d,1}\leq d^{(1+\alpha)/2}\}.$$ Indeed, for such $d$, we can take $n_1+m_1\sqrt{d}= \varepsilon_{d,1}$ and $n_2+m_2\sqrt{d}= \varepsilon_{d,1}^2$ to have $P(n_1)=dm_1^2$, $P(n_2)=dm_2^2$ and  $P(n_1)\neq P(n_2).$ From Hooley \cite{H84}, there exists $\alpha_0\in (0,\tfrac12)$ such that uniformly for $1/(\alpha_0\log D)<\alpha\le \alpha_0$ and large $D$, we have 
$$\#\{ d\leq D\,:\, \varepsilon_{d,1}\leq d^{(1+\alpha)/2}\}\gg \sqrt{D}(\alpha \log D)^2.$$ This gives the required lower bound when $D=N^{1/(1+\alpha)}$ and $\alpha\asymp 1/(\alpha_0\log D).$\end{proof}

\section{The fourth moment (quadratic case)}

When $P$ is a polynomial of degree $2$ in $\Z[X]$ we study the distribution of  $(m,n_1,n_2,n_3,n_4)\in \N^5$ that satisfy 
\begin{equation}
    \label{eq 2 be a square}
P(n_1)P(n_2)P(n_3)P(n_4)=m^2.\end{equation} We call a solution  diagonal if it satisfies $P(n_i)=P(n_j)$ for some $i\neq j$ or if $m=0.$ If we assume $n_i$ or $n_j$ sufficiently large, $P(n_i)=P(n_j)$ is equivalent to $n_i=n_j$.

According to Proposition~\ref{prop N2}, the contribution of diagonal solutions such that, for instance, $P(n_1)=P(n_2)$ and $P(n_3)\neq P(n_4)$ is bounded by $O(N^{3/2}(\log N)^2).$ So the main contribution should come from the case $P(n_i)=P(n_j)$ and $P(n_k)=P(n_\ell)$ for $\{ i,j,k,\ell\}=\{ 1,2,3,4\}.$
Our aim is to show that the contribution of off-diagonal solutions is $O(N^{2-\eta})$ with some~$\eta>0$.

We introduce the cardinality of off-diagonal solutions
$$M_4(N;P):=\#\left\{ (n_1,n_2,n_3,n_4)\in [N]^4\,: \begin{array}{l} 
P(n_1)P(n_2)P(n_3)P(n_4)=\square \neq 0\\  n_i \neq  n_j\quad (1\leq i<j\leq 4) \end{array}\right\}.$$ 
\begin{proposition}
    \label{prop N4} Let $\varepsilon>0 $ and $P$ be a separable polynomial of degree $2$ in $\Z[X]$.  For $N\geq 2$, we have
    $$M_4(N;P)\ll_{P,\varepsilon} N^{3/2+\varepsilon}.$$
\end{proposition}

\begin{proof}
As in the proof of Proposition~\ref{prop N2}, we can assume that $P(X)=X^2-\Delta$ with $\Delta\neq 0.$
This time, Lemma~\ref{mixed-4-correlation} lets us discard a bounded finite number of values of $n_j$, at the price of introducing an error term of $O_P(N^{3/2})$,
which is acceptable.

Moreover, we assume $n_j\ge N_0$, say, so that $P(n_j)>0$.
We introduce $q_1$, $q_2$, $q_3$, $q_4$ square-free integers and $m_1$, $m_2$, $m_3$, $m_4$ positive integers such that
$$P(n_j)=q_jm_j^2$$
and $q_1q_2q_3q_4$ is a square.

The following  lemma is a bound of the cardinality
\begin{equation}
    \label{eq CBbeta}C(N,d;\beta):=\# \left\{ (m,n,q)\in \N^3\,:\begin{array}{l}  q\asymp N^{\beta},\, n\leq N,\cr d\mid q,\quad \mu^2(q)=1,\quad P(n)=qm^2\end{array}\right\} .\end{equation}
    Let  $\delta :=(\log d)/\log N$.
We use the notation $$c(\beta,\delta):=\min\big\{ \beta-\delta, \max\{ \tfrac 12\beta-\delta,1- \tfrac 12\beta\}\big\}
=\begin{cases}
    \tfrac 12\beta-\delta & \text{if $1+\delta\leq \beta\leq 2,$}\\
     1-\tfrac 12\beta & \text{if $\tfrac23(1+\delta)\leq \beta\leq 1+\delta,$}\\
      \beta-\delta & \text{if $0\leq \beta\leq \tfrac 23(1+\delta).$}\\
\end{cases}$$

\begin{lemma}\label{lemma majCNdbeta}Let $P$ be a separable polynomial of degree $2$ in $\Z[X]$.  
    For $N\geq 2$, $d\ge 1$, $\beta\in [0,2]$, and $\varepsilon>0$,
    \begin{equation}
    \label{bound CBbeta}C(N,d;\beta) \ll d^\varepsilon N^{c(\beta,\delta)} (\log N)^2 .\end{equation}
    In particular, we have the following sharpening of
    Proposition~\ref{large-squares} when $\deg{P}=2$:
    \begin{equation}
        \label{large square deg=2}
    \#\{ N_0\le n\leq N \,:\, P(n)/\textnormal{core}(P(n))>N^{2\vartheta}\}\ll N^{\max\{ 2/3,1-\vartheta\}}(\log N)^2.\end{equation}
\end{lemma}

 \begin{remark}
    For $d=1$ and $P(X)=X^2-1$, bound \eqref{bound CBbeta} can be improved using \cite[theorem 1]{FJ13}. For any $\varepsilon>0$, it was proved that  $c(\beta,0)+\varepsilon$ can be replaced by $\tfrac13+\tfrac{5}{12}\beta$ if $\tfrac47\leq \beta\leq \tfrac8{11}.$
\end{remark} 
\begin{proof}
We write $\beta $ so that $q \asymp N^{\beta} $ and assume $q$ square-free.
Counting  $(n,m)$ for any $q\ll N^{\beta} $, $d\mid q$ via \eqref{bound car xy}, we obtain $$C(N,d;\beta) \ll N^{ \beta }(\log N)/d\ll N^{ \beta-\delta }(\log N).$$

If $\beta\geq \tfrac23(1+\delta),$ we use Hooley's method.
Counting $(n,q')$ such that $P(n)\equiv 0 (\bmod\, dm^2)$ and $q=q'd$ and using \eqref{rho-bound}, we get 
\begin{align*}
 C(N,d;\beta)&\ll \sum_{2^k\leq N} \sum_{m\asymp N^{1-\beta/2}/2^k}\sum_{\substack{\Omega\bmod dm^2\\ P(\Omega)\equiv 0\bmod dm^2 }}\sum_{\substack{n\equiv \Omega\bmod dm^2\\ n\ll N/2^k}}1\cr&\ll \sum_{2^k\leq N} \sum_{m\asymp N^{1-\beta/2}/2^k}\rho_P(dm^2)\Big( \frac{N}{2^kdm^2}+1\Big)
 \cr&\ll 2^{\omega(d)}\sum_{2^k\leq N} \sum_{m\asymp N^{1-\beta/2}/2^k}2^{\omega(m)}\Big( \frac{N}{2^kdm^2}+1\Big)\cr&\ll  d^\varepsilon N^{\max\{  \beta/2-\delta,1-  \beta/2\}} (\log N)^2 .
 \end{align*}
This finishes the proof of \eqref{bound CBbeta}.

To deduce \eqref{large square deg=2} from \eqref{bound CBbeta}, we write $P(n)=qm^2$ with $\mu^2(q)=1$. 
We count $n\in (\tfrac12 N,N]$ such that $P(n)\asymp N^2.$ The counting for $n\le N$ can be easily deduced.
We have  $q\ll N^{2-2\vartheta}$ so that  
we get 
$\beta\leq 2-2\vartheta+O(1/\log N).$ 

We can reduce to $O(\log N)$ values of $\beta$. 
 For $\beta\le\tfrac23$ one has $c(\beta,0)\le\beta$ and a summation over $\beta$ gives the required bound. For $\tfrac23<\beta\le1$, we use   $c(\beta,0)\le1-\tfrac 12\beta$. For $1<\beta\le2-2\vartheta$ one has $c(\beta,0)=\tfrac12\beta \le1-\vartheta$.  This gives \eqref{large square deg=2}.
%
 %
\end{proof}

We parameterize $q_j$ by $d,$ $\{d_{ij}\}_{1\leq i<j\leq 4}$ such that $\mu^2(dd_{12}d_{13}d_{14}d_{23}d_{24}d_{34})=1$ and 
\begin{equation}
    \label{def dij}
 q_1=dd_{12}d_{13}d_{14} ,\,
    q_2=dd_{12}d_{23}d_{24} ,\,
    q_3=dd_{13}d_{23}d_{34} ,\,
    q_4=dd_{14}d_{24}d_{34} . \end{equation} 
    We have $d=\gcd_j (q_j)$ and $d_{ij}=\gcd (q_i/d,q_j/d)$. 
To get a clearer picture, we introduce $\alpha\geq 0$ and $\alpha_{ij}\geq 0$ such that $d\asymp N^\alpha$ and $d_{ij}\asymp N^{\alpha_{ij}}$. We write $\bfalpha=(\alpha, \alpha_{12},  \alpha_{13},  \alpha_{14},  \alpha_{23},  \alpha_{24},  \alpha_{34})$ and
$\bfbeta=(\beta_1,\beta_2,\beta_3,\beta_4)$.
For simplicity, we impose the following
\begin{equation}
    \label{rel alpha beta}
\begin{cases}
    \beta_1=\alpha +\alpha_{12}+\alpha_{13}+\alpha_{14} ,\\ 
    \beta_2=\alpha +\alpha_{12}+\alpha_{23}+\alpha_{24} ,\\
    \beta_3=\alpha +\alpha_{13}+\alpha_{23}+\alpha_{34} ,\\
    \beta_4=\alpha +\alpha_{14}+\alpha_{24}+\alpha_{34} .\end{cases}\end{equation}

We consider
    $$
    C(N;\bfalpha,\bfbeta):=\#\left\{(\bfm,\bfn,\bfq)\in \N^{12}:\begin{array}{l} d,(d_{ij})_{1\leq i<j\leq 4} \text{ as in } \eqref{def dij} \text{ with } d\asymp N^\alpha,\, d_{i,j}\asymp N^{\alpha_{ij}}\cr q_j\asymp N^{\beta_j},\quad n_j\leq N,\quad \mu^2(dd_{12}d_{13}d_{14}d_{23}d_{24}d_{34})=1,\cr P(n_j)=q_jm_j^2,\, P(n_i)\neq P(n_j)\quad (1\leq i<j\leq 4) \end{array}\!\!\right\}.$$

    The following lemma is the key stage in our argument.
\begin{lemma}
    \label{lemma maj CBalphabeta} Let $\varepsilon>0$, $P$ be a separable polynomial of degree $2$ in $\Z[X]$.
    For any $\bfalpha,\bfbeta$ such that \eqref{rel alpha beta} and $\bfbeta\in [0,2]^4$, and $N\geq 2$, we have 
\begin{equation}
    \label{bound C}
C(N;\bfalpha,\bfbeta)\ll N^{3/2+\varepsilon}.\end{equation}
\end{lemma}

\begin{proof}
We give three individual Diophantine estimates,
and numerically optimize the resulting exponents over all dyadic ranges.
The individual Diophantine estimates are due to us.
From there, we originally derived a suboptimal bound of $N^{9/5+\varepsilon}$, which sufficed for our main CLT results.
GPT-5.5 Pro then showed us that the exponent could be improved to $3/2+\varepsilon$ by numerical optimization (linear programming).

For notational convenience, let $\alpha_{ij}=\alpha_{ji}$.
Using the piecewise linear function $c(\beta,\delta)$ defined before
Lemma~\ref{lemma majCNdbeta}, let
\begin{align}
R &:= (\beta_1+\beta_2+\beta_3+\beta_4)/2, \\
S_{ijk} &:=
c(\beta_i,0)
+ c(\beta_j,\alpha+\alpha_{ij})
+ c(\beta_k,\alpha+\alpha_{ik}+\alpha_{jk}),
\label{def ebeta} \\
T_{ij} &:=
c(\beta_i,0)
+ c(\beta_j,\alpha+\alpha_{ij})
+ 2-R+(\beta_i+\beta_j)/2,
\end{align}
where we emphasize that $S_{ijk}$ and $T_{ij}$ depend on the order of $i,j,k$.
We claim that
\begin{equation}
C(N;\bfalpha,\bfbeta)\ll N^{E+\varepsilon},
\end{equation}
where
\begin{equation}
E := \min\left\{R,
\min_{\substack{i,j,k\in [4]\\ \mathrm{distinct}}} S_{ijk},
\min_{\substack{i,j\in [4]\\ \mathrm{distinct}}} T_{ij}
\right\}.
\end{equation}

\textbf{First bound.}
If we fix  $q_j$, then the product $q_1q_2q_3q_4\ll N^{\beta_1+\beta_2+\beta_3+\beta_4}$ is a square. By~\eqref{bound car xy}, the number of $(n_1,\ldots,n_4)$ is $\ll (\log N)^4$. So, the contribution is
$$C(N;\bfalpha,\bfbeta)\ll\sum_{q^2\ll N^{\beta_1+\beta_2+\beta_3+\beta_4}}\tau_4(q^2)(\log N)^4
\ll N^{(\beta_1+\beta_2+\beta_3+\beta_4)/2}(\log N)^{13}.$$ Here we use
$\tau_4(p^2)=10.$
Thus $C(N;\bfalpha,\bfbeta)\ll N^{R+\varepsilon}$.

\textbf{Second bound.}
If $q_i$ is fixed, to count $q_j$ we have the additional condition $dd_{ij} \mid q_j$.
If $q_i$ and $q_j$ are  fixed, to count $q_k$ we have the additional condition $dd_{ik}d_{jk} \mid q_k.$ 
Introducing \eqref{def ebeta}
by Lemma~\ref{lemma majCNdbeta},
we obtain
$C(N;\bfalpha,\bfbeta)\ll N^{S_{ijk}+\varepsilon}$
for any $\varepsilon>0$,
by the divisor bound.

\textbf{Third bound.}
Let $(i,j) = (4,3)$.
We use Lemma~\ref{lemma majCNdbeta}
to first choose $(q_4,n_4,m_4)$,
then $(q_3,n_3,m_3)$.
Using the divisor bound, we then choose $d_{1t},d_{2t}$ for $t=3,4$.
Finally, trivially choose $m_1\ll N^{1 -\beta_1/2 }$ and $m_2\ll N^{1 -\beta_2/2 }$.
The pair $(n_1,n_2)$
satisfies $$d_{23}d_{24}m_2^2 P(n_1)-d_{13}d_{14}m_1^2 P(n_2)=0,$$
where the coefficients in front of $P(n_1),P(n_2)$ are $\ll N^{O(1)}$.
We have $d_{23}d_{24}m_2^2 \neq d_{13}d_{14}m_1^2 $ since $P(n_1)\neq  P(n_2).$
The number of such pairs is then $\ll N^\varepsilon$
by Lemma~\ref{CS2.5}.
We get
$C(N;\bfalpha,\bfbeta)\ll N^{T_{43}+\varepsilon}$,
since $2-(\beta_1+\beta_2)/2 = 2-R+(\beta_4+\beta_3)/2$.

It remains only to prove that \(E\le 3/2\).
This is done in Appendix~\ref{code}.
\end{proof}

\begin{remark}
Gemini 3.1 Pro has a short proof that $E\le 21/11$.
First, $c(\beta, \delta) \le 1/2 + \beta/4 - \delta/2$ for $\beta \le 2$ and $\delta \ge 0$.
Averaging over $i,j,k$ using $\mathbb{E}[\beta_i] = R/2$
and $\mathbb{E}[\alpha_{ij}] = R/6 - \alpha/3$ gives
$\mathbb{E}[S_{ijk}] \le 3/2 + R/8$
and
$\mathbb{E}[T_{ij}]\le 3 - R/3$,
since $\alpha\ge 0$.
The worst case is $R = 36/11$.
\end{remark}

Lemma~\ref{lemma maj CBalphabeta} implies Proposition~\ref{prop N4} since the number of different $\alpha,$ $\bfalpha$ and $\bfbeta$ satisfying~\eqref{rel alpha beta} is bounded  by $O((\log N)^{7}).$\end{proof}

\section{Polynomials of degree two}
In this section, we derive a central limit theorem for random completely multiplicative functions with polynomial phase $P(n)$, where the polynomial  $P $ is any given polynomial belonging to the family of polynomials that we studied in the previous section.

Our work in the previous two sections now allows us to establish the following central limit theorem for random completely multiplicative functions along polynomial phases.

\begin{theorem}\label{thm: CLT}
   Let $P$ be a separable polynomial of degree $2$ in $\Z[X]$ with positive leading coefficient. Let $g$ be an extended Rademacher RMF. Then as $N\to +\infty$,
  \begin{equation}
      \label{eq CLT}
  \frac{1}{\sqrt{N}} \sum_{1\leq n \leq N} g(P(n)) \xrightarrow{d} \mathcal{N}(0,1).
  \end{equation}
\end{theorem}
 By taking $a=a_N$ in Theorem~\ref{theorem Extended Rademacher}, we deduce the following proposition.

\begin{proposition}[Deduced from the McLeish central limit theorem]\label{prop: mcleish}
Theorem~\ref{thm: CLT} holds if the following is true,
where the notation $P^*$ is defined in \eqref{def P*}.
\begin{align*}
&  (1)\qquad\#\left\{( n_1, n_2 )\in [N]^{2}: \begin{array}{l}P(n_1)P(n_2)  = \square \neq 0\cr n_1\neq n_2 \end{array}  \right\} = o_{N\to +\infty}(N ), \cr
  &  (2)\qquad\#\left\{( n_1, n_2, n_3, n_4)\in [N]^{4}: \begin{array}{l}P(n_1)P(n_2) P(n_3)P(n_4)= \square \neq 0\cr n_i\neq n_j\quad (1\leq i<j\leq 4)\end{array}  \right\} = o_{N\to +\infty}(N^{2}), \cr
   & (3) \qquad
\# \left\{ ( n_1, n_2, n_3, n_4) \in [N]^{4}:  \begin{array}{l}
 P^{*}(P(n_i)) = P^{*}(P(n_j)) \quad (1\leq i<j\leq 4),\\   P(n_1)P(n_2) P(n_3)P(n_4)   =\square \neq 0\end{array}\right\} =o(N^{2}).     
\end{align*}
\end{proposition}

\begin{proof}First of all, 
given a polynomial, we notice that there exists $\varepsilon\in \{ \pm1\}$ and a constant~$N_0$ such that~$\varepsilon P(n)>0$ for any $n>N_0$. We assume $\varepsilon=1$. The limiting distribution of~$\frac{1}{\sqrt{N}}\sum_{1\le n \le  N}g(P(n))$ is the same as that of~$\frac{1}{\sqrt{N}}\sum_{N_0< n \le N}g(P(n))$. Thus, we may assume that $P(n)> 0$. The rest of the deduction is the same as in the previous work. The version of random multiplicative functions is covered by Klurman--Shkredov--Xu \cite[Theorem 1.1]{KSX23} (in the Steinhaus case) if $P$ isn't a scalar multiple of a power of a linear polynomial and by Chinis and Shala \cite[Theorem 1.2]{CS25}  (in the Rademacher case) if $P$ is a product of at least two distinct linear factors $\in \Z[X]$ or irreducible of degree $2$. 

We check the conditions (1), (2) and (3) of Theorem~\ref{theorem Extended Rademacher}
with $a=a_N$ and
$$S:=\{ \co(P(n))\,: 1\le n\le N\} \smallsetminus \{1\}.$$
Condition (1) of Theorem~\ref{theorem Extended Rademacher}  is a consequence of
condition  (1)  since it implies
$$\sum_ma(m)^2=\#\left\{( n_1, n_2 )\in [N]^{2}:  P(n_1)P(n_2)  = \square \neq 0  \right\} =N(1+o(1)).$$  Condition (2) of Theorem~\ref{theorem Extended Rademacher}  is equivalent to 
condition  (3). Condition (3) of Theorem~\ref{theorem Extended Rademacher}  is equivalent to 
condition  (2).
\end{proof}

Now we complete the proof of Theorem~\ref{thm: CLT}.

\begin{proof}[Proof of Theorem~\ref{thm: CLT}]
    By Proposition~\ref{prop: mcleish}, we only need to verify the three conditions. Condition (1) follows from Proposition~\ref{prop N2} whereas condition (2) follows from Proposition~\ref{prop N4}. 



    To prove condition (3), we roughly follow \cite{KSX23}.
    However, we refine the strategy to give a reasonable quantitative estimate.
    First, Proposition~\ref{prop N4} implies that the off-diagonal contribution is $o(N^{2})$, even with a power saving in $N$.
It remains to bound the diagonal contribution.
Let $\LL(N):=\exp\{ \sqrt{(\log N)(\log\log N)}\}.$
We will show that
\begin{equation}
\label{diagonal-martingale}
S_2(N):=\# \big\{ (  n_1, n_2) \in [N]^{2}:  P^{*}(P(n_1)) = P^{*}(P(n_2))\big\}
= O\big(N^2\LL (N)^{-2+o(1)}\big),
\end{equation}
which is $o(|S|^2)$ for any set $S\subset [N]$ satisfying \eqref{thinner}
in the remark after Theorem~\ref{extended-rademacher-P}.

We denote
\begin{align*}
S_2(N,p)&:=\# \big\{ (  n_1, n_2) \in [N]^{2}:  P^{*}(P(n_1)) = P^{*}(P(n_2))=p\big\}
\cr&
=\Big(\# \big\{ n \in [N] :  P^{*}(P(n )) = p\big\}\Big)^2 .    
\end{align*}

To prove \eqref{diagonal-martingale}, we divide into three cases: $p \le  \LL(N)^2$, $ \LL(N)^2 < p \le N $ or $p> N$, where $p: = P^{*}(P(n))$. The second and third cases are a special case of the Ekedahl sieve
in the form of \cite[Theorem 3.3]{BhargavaEkedahl},
but we write out details for the reader's convenience. \\

We begin with the case $p \le  \LL(N)^2$.
To bound $\# \big\{ n \in [N] :  P^{*}(P(n )) = p\big\}$
we write $P(n)=dm^2$ with $\mu^2(d)=1$ and $p\mid d.$ We have 
$P^{*}(P(n))=P^{+}(d).$
First if $d\leq N$, by Lemma~\ref{siegel} there exists $\kappa=\kappa_P$ such that 
$\#\{ (n,m):P(n)=dm^2\}\ll \kappa^{\omega(d)}.$\footnote{This is assuming $\deg P\ge 3$, but the case $\deg P=2$ is simpler and left as an exercise to the reader.}
Then
$$\# \big\{ n \in [N] :  P^{*}(P(n )) = p, d\leq N\big\}
\ll \sum_{\substack{d\leq N\\ P^+(d)=p}}\kappa^{\omega(d)}
\ll   \sum_{\substack{d\leq N/p\\ P^+(d)\leq p}}\kappa^{\omega(d)}.$$
By  Drappeau \cite[Theorem~1]{D16}
and friable-number estimates such as \cite[(1.12)]{Granville},
we get
$$\# \big\{ n \in [N] :  P^{*}(P(n )) = p, d\leq N\big\}
\ll \frac{N}{p}(\log p)^{\kappa-1} \big( \exp\{ -u_p (\log u_p)(1+o(1))\}+\LL(N)^{-2}\big),$$
where $u_p:=(\log N)/(\log p).$ The second term accounts for the case where $p$ is too small in terms of $N$.
Secondly, if $d>N$, since $d$ is $p$-friable, we can construct a divisor $d'$ of $d$ such that $p\mid d'$ and $N/p<d'\leq N$. 
We get
\begin{align*}
    \# \big\{ n \in [N] :  P^{*}(P(n )) = p, d> N\big\}&\ll \sum_{\substack{N/p<d'\leq N\\ P^+(d')=p}}\#\{ n\in [N]: d'\mid P(n)\}
\cr&\ll\frac{N}{p} \sum_{\substack{N/p^2<d'\leq N/p\\ P^+(d')\leq p}}\frac{(\deg P)^{\omega(d') }}{d'}\end{align*}
where we used \eqref{rho-bound} to bound $\rho_P(d').$
As in the case $d\leq N$ we get
$$ \# \big\{ n \in [N] :  P^{*}(P(n )) = p, d> N\big\}
\ll \frac{N}{p}(\log p)^{\deg P} \big( \exp\{ -u_p (\log u_p)(1+o(1))\}+\LL(N)^{-2}\big).$$
We sum over $p$ such that $p\asymp \LL(N)^{\alpha} $ with $\alpha \leq 2.$ We have 
$$u_p (\log u_p)=\sqrt{ (\log N)(\log\log N)}(1/2\alpha+o(1)).$$
This implies
$$\sum_{p\asymp \LL(N)^{\alpha}}
S_2(N,p)\ll N^2\big( \LL(N)^{-\alpha-1/\alpha+o(1)}+ \LL(N)^{-2}\big)$$
Since $\min_{\alpha\in [0,2]} \{ \alpha+1/\alpha\}=2$, we obtain
$$\sum_{p\leq \LL(N)^{2}}
S_2(N,p)\ll N^2\LL(N)^{-2+o(1)}
.$$
Next, we consider the case $ \LL(N)^2 < p \le N$. Notice that the number of $n\le p$ such that $p|P(n)$ is at most $d=\deg P$ for each fixed $p$ and consequently, the number of {\it diagonal} solutions is at most 
\[\ll_P  \sum_{\LL(N)^2< p \le N}  \frac{N^{2}}{p^{2}}   = O_P\big(N^{2}\LL(N)^{-2}\big). \]
Finally, if $p> N$ we notice that for each fixed $n\le N$ with large $N\ge 1,$ there are at most~$O (1)$ primes $p\ge N$ with $p\vert P(n)$ and therefore there is in total $O (N)$ number of pairs~$(p,n)$ such that $p\vert P(n)$ and $p\ge N$. Combining with the fact that for each $p\ge N$ there are at most~$O (1)$ integers $n$ such that $p\vert P(n)$, it follows that the number of diagonal solutions in this regime is  $$\ll\sum_{p>N} S_2(N,p)\ll \sum_{p>N}\sum_{\substack{n_1\leq N\\ p\mid P(n_1) }}\sum_{\substack{n_2\leq N\\ p\mid P(n_2)}}1\ll \sum_{\substack{n_1\leq N }}\sum_{\substack{p\mid P(n_1)\\  p>N}}1\ll N$$ which is negligible. This concludes the proof.
\end{proof}

\section{Polynomials of higher degree}

Fix a polynomial $P\in \mathbb{Z}[X]$ with no repeated roots.
Assume $\deg{P} = r \ge 3$.
Moreover, assume $P(x)\ge \frac12 x^r$ and $P'(x)\ge 1$ for all $x\ge N_0$,
for some integer $N_0\ge 1$.

 Given a set $S$, let $\binom{S}{k}$ denote the set of $k$-element subsets of $S$.
The case $A=1$ of the following result
produces an $o(N^2)$ off-diagonal estimate
for the fourth moment of a Rademacher RMF along values of $P$.
It remains $o(N^2)$ for $A\le N^c$ for any fixed $c<4/11$.

\begin{proposition}
\label{unconditional-Rademacher-off-diagonal}
Let $N\ge N_0$ and $1\le A\le N^5$.
Let $M_4(N,A)$ be the number of sets
$$\{n_1,n_2,n_3,n_4\}\in \binom{[N_0,N]}{4}$$
such that $P(n_i)/\co(P(n_i)) \le A$ for all $1\le i\le 4$,
and $\prod_{1\le i\le 4} P(n_i)$ is a square.
Then
$$M_4(N,A) \ll_{P,\varepsilon} A^{11/9} N^{14/9+\varepsilon}.$$
\end{proposition}

\begin{proof}
For any tuple $(n_1,n_2,n_3,n_4)$ counted by $M_4(N,A)$, we may write
$$\co(P(n_i)) = d \prod_{j\ne i} d_{ij}$$
for all $1\le i\le 4$,
where $$d = \gcd_{1\le k\le 4}(\co(P(n_k)))$$
and $$d_{ij} = d_{ji} = \gcd(\co(P(n_i))/d, \co(P(n_j))/d).$$

Suppose $n_i\asymp N_i$, $d\asymp D$, and $d_{ij}\asymp D_{ij}$.
Then $$\frac{N_i^r}{A} \ll D \prod_{j\ne i} D_{ij} \ll N_i^r$$
for all $1\le i\le 4$.
By permuting the indices $2,3,4$, we may assume that
$$D_{12} = \max(D_{12},D_{13},D_{14}).$$
By further permuting $\{3,4\}$ if necessary, we may also assume that
$$N_3 \ge N_4.$$
In particular,
$$(D D_{12})^3
\gg D D_{12}^3
\ge D D_{12} D_{13} D_{14} \gg \frac{N_1^r}{A}$$
and
$$DD_{13}D_{23} \gg \frac{N_3^r}{A D_{34}}
\gg \frac{(N_3+N_4)^r}{A D_{34}}.$$

By the divisor bound
on $P(n_1)$ and $P(n_2)$,
the Chinese remainder theorem modulo the square-free integers
$dd_{12}$ and $dd_{13}d_{23}$,
and Lemma~\ref{siegel} for the curve
$P(x) = \co(P(n_4)) y^2$,
the contribution to $M_4(N,A)$ from given dyadic parameters $N_i,D,D_{ij}$ is
$$M_4(N,A;N_i,D,D_{ij})
\ll N^\varepsilon N_1
\left(\frac{N_2}{DD_{12}} + 1\right)
\left(\frac{N_3}{DD_{13}D_{23}} + 1\right),$$
by choosing $n_1,n_2,n_3,n_4$ in that order,
and noting that $\co(P(n_4))$ is uniquely determined by $P(n_1)P(n_2)P(n_3)$.
Let $0\le \delta\le 1$.
Since $r\ge 3$, we have $DD_{12} \gg N_1/A^{1/3}$, so
\begin{equation*}
\begin{split}
M_4(N,A;N_i,D,D_{ij})
&\ll N^\varepsilon A^{1/3} (N_2 + N_1)
\left(\frac{AD_{34}N_3}{(N_3+N_4)^r} + 1\right) \\
&\ll N^\varepsilon A^{1/3} (N_2 + N_1)
A(N_3+N_4)^{1-\delta}
\ll A^{4/3} N^{2-\delta+\varepsilon},
\end{split}
\end{equation*}
provided that $D_{34} \le (N_3+N_4)^{r-\delta}$.

Suppose next that $D_{34} \ge (N_3+N_4)^{r-\delta}$.
Then for $k=3,4$ we have $$DD_{1k}D_{2k}A_k \asymp \frac{N_k^r}{D_{34}}
\ll (N_3+N_4)^\delta,$$
if we let $A_i \asymp P(n_i)/\co(P(n_i))$ for $1\le i\le 4$.
However, we have
$$\frac{\co(P(n_1))}{d_{13}d_{14}} - \frac{\co(P(n_2))}{d_{23}d_{24}}
= 0
= \frac{\co(P(n_3))}{d_{13}d_{23}} - \frac{\co(P(n_4))}{d_{14}d_{24}}.$$
Therefore,
fixing the values of
$$d_{13},d_{14},d_{23},d_{24},P(n_i)/\co(P(n_i)),$$
and applying Lemma~\ref{CS2.5},
we obtain the bound
\begin{equation*}
\begin{split}
M_4(N,A;N_i,D,D_{ij})
&\ll \sum_{\substack{d_{13},d_{14},d_{23},d_{24} \\
A_1,A_2,A_3,A_4}} (A_1A_2A_3A_4)^{1/2}
(N_1+N_2)^{1/3+\varepsilon} (N_3+N_4)^{1/3+\varepsilon} \\
&\ll (N_3+N_4)^{2\delta} A
(N_1+N_2)^{1/3+\varepsilon} (N_3+N_4)^{1/3+\varepsilon}
\ll A N^{2/3+2\delta+2\varepsilon}.
\end{split}
\end{equation*}

Taking $N^\delta = A^{1/9} N^{4/9}$, assuming $A\le N^5$,
we conclude that in every case above,
$$M_4(N,A;N_i,D,D_{ij}) \ll A^{11/9} N^{14/9+\varepsilon}.$$
Dyadic summation over $N_i,D,D_{ij}$ completes the proof.
\end{proof}

\begin{remark}Congruences modulo $dd_{12}$ and $dd_{13}d_{23}$
were similarly used in the derivation of the exponent \eqref{def ebeta} when $\deg{P}=2$.
In the argument above, we have made some effort to optimize the exponent of $N$,
but the exponent of $A$ may be suboptimal.
\end{remark}

Let $P$ be as above.
The variance (second moment)
analog of Proposition~\ref{unconditional-Rademacher-off-diagonal} is simpler.

\begin{proposition}
\label{easy-variance}
Let $A\ge 1$ and $N\ge N_0$.
Let $M_2(N,A)$ be the number of sets
$$\{n_1,n_2\}\in \binom{[N_0,N]}{2}$$
for which $P(n_i)/\co(P(n_i)) \le A$ for all $1\le i\le 2$,
and $P(n_1)P(n_2)$ is a square.
Then
$$M_2(N,A)\ll_\varepsilon A N^{1/3+\varepsilon}.$$
\end{proposition}

\begin{proof}
Write $P(n_i) = dm_i^2$ where $m_i^2\le A$.
Then $m_2^2 P(n_1) = m_1^2 P(n_2)$.
Since $n_1\ne n_2$, it follows that $m_1^2\ne m_2^2$,
and $M_2(N,A) \ll \sum_{m_1,m_2} N^{1/3+\varepsilon}
\ll (A^{1/2})^2 N^{1/3+\varepsilon}$
by Lemma~\ref{CS2.5}.
\end{proof}

\begin{remark}
Proposition~\ref{easy-variance} and its proof
hold even if $\deg P = 2$.
\end{remark}

Recall the definition of H$\delta$ from \eqref{SFSC}.
For each set $S\subseteq [N_0,N]$, let
\begin{equation}
\label{define-V(S)}
V(S) := \#\{(n_1,n_2)\in S^2: P(n_1)P(n_2) = \square\}.
\end{equation}

\begin{theorem}
\label{THM:conditional-CLT}
Assume that H$\delta$ holds for some constant $\delta>0$.
Then
\begin{equation}
\label{conditional-variance}
V([N_0,N]) \sim N
\end{equation}
as $N\to \infty$.
Moreover, if $g$ is an extended Rademacher RMF then
\begin{equation}
\label{EQN:conditional-CLT-2}
\frac{1}{\sqrt{N}} \sum_{1\le n\le N} g(P(n))
\xrightarrow[]{d} \mathcal{N}(0,1).
\end{equation}
\end{theorem}

\begin{proof}
Let $A := N^{1/3}$.
Let
\begin{equation*}
\begin{split}
\mathcal{S} &:= \left\{n\in [N_0,N]:
\co(P(n)) \ne 1
\qquad \textnormal{and} \qquad
\max_{\substack{m\in [N_0,N] \\ \co(P(m))=\co(P(n))}}
\left(\frac{P(m)}{\co(P(m))}\right) \le A\right\} \\
&\subseteq \left\{n\in [N_0,N]:
\frac{P(n)}{\co(P(n))} \le A\right\}.
\end{split}
\end{equation*}
The condition for an integer $n\in [N_0,N]$
to lie in $\mathcal{S}$ depends only on $\co(P(n))$.
Therefore,
\begin{equation*}
V([N_0,N]) - V(\mathcal{S})
= V([N_0,N]\smallsetminus \mathcal{S})
\le F([N_0,N]\smallsetminus \mathcal{S})
\ll_{P,\varepsilon} N^{\varepsilon}\, |[N_0,N]\smallsetminus \mathcal{S}|,
\end{equation*}
by Lemma~\ref{mixed-2-correlation}.
However, $V(\mathcal{S}) = |\mathcal{S}| + O_\varepsilon(A N^{1/3+\varepsilon})$
by Proposition~\ref{easy-variance}.
Since $$|[N_0,N]\smallsetminus \mathcal{S}|
\ll_{P,\varepsilon} N^\varepsilon \left(1 + \frac{N}{A^{1/2-\varepsilon}} + N^{1-\delta}\right)$$
by Lemmas~\ref{siegel} and~\ref{easy-sieve}
and our assumption \eqref{SFSC},
it follows that
\begin{equation*}
V(\mathcal{S}) - N
\ll_{P,\varepsilon} \frac{N^{1+\varepsilon}}{A^{1/2}} + N^{1-\delta} + A N^{1/3+\varepsilon}
= o(N)
\end{equation*}
and
\begin{equation*}
V([N_0,N]) - N
\ll_{P,\varepsilon} \frac{N^{1+\varepsilon}}{A^{1/2}} + N^{1-\delta} + A N^{1/3+\varepsilon}
= o(N).
\end{equation*}
This establishes the variance claim, \eqref{conditional-variance}.

The CLT
\begin{equation}
\label{EQN:conditional-CLT-1}
\frac{1}{\sqrt{N}} \sum_{n\in [N_0,N]} g(P(n))
= \frac{1}{\sqrt{N}} \sum_{n\in [1,N-N_0+1]} g(P(n+N_0-1))
\xrightarrow[]{d} \mathcal{N}(0,1)
\end{equation}
follows from
the axiomatic framework of Theorem~\ref{theorem Extended Rademacher},
where we take
$a=a_{N-N_0+1}$, and
$$S := \{\co(P(n)): n\in \mathcal{S}\}.$$
Indeed, a short calculation using the definitions of $\mathcal{S}$ and $S$
shows that
$$\sum_{d\in \Z\smallsetminus S} a(d)^2 = V([N_0,N]\smallsetminus \mathcal{S}) = o(N),
\qquad \sum_{d\ge 1} a(d)^2 = V([N_0,N]) \sim N,$$
and the required fourth moment estimates hold unconditionally by Proposition~\ref{unconditional-Rademacher-off-diagonal}
and \eqref{diagonal-martingale},
since $A^{11/9} N^{14/9+\varepsilon} = o(N^2)$.
\end{proof}

\begin{remark}
Proposition~\ref{large-squares} implies that H$\delta$ holds if
 every irreducible factor of $P$ has degree~$\le~3$.
The function field analog of H$\delta$ is also known unconditionally for arbitrary square-free $P$;
this follows from Ramsay \cite{Ramsay}
and Poonen \cite{Poonen}.
Finally, we emphasize that there is hope of proving \eqref{conditional-variance}
or \eqref{EQN:conditional-CLT-2} unconditionally, without assuming H$\delta$.
For example, suppose $P(X) = b^2 Q(X)^4 - 1$ where $b\ge 2$ is square-free
and $Q\in \mathbb{Z}[X]$ with $\deg{Q}\ge 1$.
Then by
Bennett--Walsh \cite{BennettWalsh},
we have $V([N_0,N]) - |[N_0,N]| = 0$.
It is also known that if $b=1$, then $V([N_0,N]) - |[N_0,N]| \ll 1$.
For all $b\ge 1$, it follows that \eqref{conditional-variance} holds unconditionally.
\end{remark}



\appendix

\section{Numerical optimization}
\label{code}

\subsection*{Code}
Below, we present a copy of AI-generated, human-verified
\href{https://pypi.org/project/z3-solver/}{Z3Py} code
available at
\begin{center}
\url{https://github.com/wangyangvictor/quad4prod}
\end{center}
to prove $E\le 3/2$ in Lemma~\ref{lemma maj CBalphabeta}.
Z3Py is a Python-based implementation of Microsoft's Z3 solver,
which includes linear programming.
We ran the code online
using \href{https://colab.research.google.com/}{Google Colab}.

\begin{tiny}
\begin{verbatim}
!pip -q install z3-solver       # Installs Z3 in Google Colab.
from itertools import permutations
from z3 import If, Real, RealVal, Solver, sat, unsat

def rational(numerator, denominator):
    return RealVal(numerator) / RealVal(denominator)

THREE_HALVES = rational(3, 2)

# Variables: a is alpha, and aij is alpha_{ij}.
a = Real("alpha")
a12, a13, a14 = Real("alpha12"), Real("alpha13"), Real("alpha14")
a23, a24, a34 = Real("alpha23"), Real("alpha24"), Real("alpha34")
alpha_variables = [a, a12, a13, a14, a23, a24, a34]

def A(i, j):
    if i > j:
        i, j = j, i
    return {
        (1, 2): a12, (1, 3): a13, (1, 4): a14,
        (2, 3): a23, (2, 4): a24, (3, 4): a34,
    }[(i, j)]

def beta(i):
    return a + sum(A(i, j) for j in (1, 2, 3, 4) if j != i)

def c(beta_value, delta):
    # This is exactly
    #   c(beta,delta)=min{beta-delta, max{beta/2-delta, 1-beta/2}}.
    # It is written using the Z3Py If-Then-Else function If(A,B,C).
    return If(
        beta_value >= 1 + delta,
        beta_value / 2 - delta,
        If(3 * beta_value >= 2 * (1 + delta), 1 - beta_value / 2, beta_value - delta),
    )

B = {i: beta(i) for i in (1, 2, 3, 4)}
R = sum(B.values()) / 2

S_candidates = []
for i, j, k in permutations((1, 2, 3, 4), 3):
    S_candidates.append(c(B[i], 0) + c(B[j], a + A(i, j)) + c(B[k], a + A(i, k) + A(j, k)))

T_candidates = []
for i, j in permutations((1, 2, 3, 4), 2):
    T_candidates.append(c(B[i], 0) + c(B[j], a + A(i, j)) + 2 - R + (B[i] + B[j]) / 2)

# Admissible region: alpha, alpha_ij >= 0 and 0 <= beta_i <= 2.
admissible = [x >= 0 for x in alpha_variables]
admissible += [B[i] >= 0 for i in B]
admissible += [B[i] <= 2 for i in B]

# To prove E <= 3/2, ask Z3 whether the opposite is possible: R, every S_{ijk} and T_{ij} all exceed 3/2.
bad = Solver()
bad.add(admissible)
bad.add(R > THREE_HALVES)
bad.add([S > THREE_HALVES for S in S_candidates])
bad.add([T > THREE_HALVES for T in T_candidates])

bad_result = bad.check()
print("Is it possible for E > 3/2?", bad_result)     # Note: sat means yes, unsat means no.

assert bad_result == unsat     # The program stops running if E > 3/2 is possible.

# Optional sharpness check: equality is attained by the linear program.
sharp = Solver()
sharp.add(admissible)
sharp.add(R >= THREE_HALVES)
sharp.add([S >= THREE_HALVES for S in S_candidates])
sharp.add([T >= THREE_HALVES for T in T_candidates])

print("Sharpness check:", sharp.check())
print(sharp.model())
\end{verbatim}
\end{tiny}

\subsection*{Output}
The code compiles and gives the following output:
\begin{tiny}
\begin{verbatim}
Is it possible for E > 3/2? unsat
Sharpness check: sat
[alpha23 = 3/16, alpha34 = 5/4, alpha13 = 1/16, alpha = 0,
 alpha12 = 5/4, alpha14 = 3/16, alpha24 = 1/16]
\end{verbatim}
\end{tiny}

\begin{remark}
GPT-5.5 Pro claims that $E = 3/2$ occurs precisely
along the one-parameter family
\[
\alpha=0,
\qquad
\alpha_{14}=\alpha_{23}=5/4,
\qquad
\alpha_{12}=\alpha_{34}=u,
\qquad \alpha_{13}=\alpha_{24}=1/4-u,
\]
up to permutation,
where $0\le u\le 1/4$.
We do not need this fact, so we do not prove it.
\end{remark}

\end{document}